\documentclass[11pt]{amsart}
\usepackage{amsmath, amssymb, amsthm, latexsym}
\usepackage{amsthm}
\usepackage{fullpage}
\usepackage{amssymb}
\usepackage{enumerate}
\usepackage{hyperref}
\usepackage{enumitem}

\newcommand{\legendre}[2]{%
\genfrac{(}{)}{}{}{#1}{#2}}

\newtheorem{theorem}{Theorem}[section]

\newtheorem{lemma}[theorem]{Lemma}
\newtheorem{conjecture}[theorem]{Conjecture}

\theoremstyle{definition}

\newtheorem{remark}[theorem]{Remark}
\newtheorem{example}[theorem]{Example}

\DeclareMathOperator{\ord}{ord}

\DeclareMathOperator{\Gal}{Gal}

\DeclareFontFamily{U}{wncy}{}
\DeclareFontShape{U}{wncy}{m}{n}{<->wncyr10}{}
\DeclareSymbolFont{mcy}{U}{wncy}{m}{n}
\DeclareMathSymbol{\Sh}{\mathord}{mcy}{"58}

\usepackage[T2A,T1]{fontenc}
\DeclareSymbolFont{cyrillic}{T2A}{cmr}{m}{n}
\DeclareMathSymbol{\Sha}{\mathalpha}{cyrillic}{216}

\usepackage{enumitem}

\setlist[enumerate]{leftmargin=*}

\usepackage{array}

\title{Prime twists of elliptic curves}
\begin{document}
\author[Daniel Kriz]{Daniel Kriz}\email{dkriz@princeton.edu}
\address{Department of Mathematics, Princeton University, Fine Hall, Washington Rd, Princeton, NJ 08544}
\author[Chao Li]{Chao Li}\email{chaoli@math.columbia.edu} 
\address{Department of Mathematics, Columbia University, 2990 Broadway,
 New York, NY 10027}


\date{\today}

\maketitle

\begin{abstract}
For certain elliptic curves $E/\mathbb{Q}$ with $E(\mathbb{Q})[2]=\mathbb{Z}/2 \mathbb{Z}$, we prove a criterion for prime twists of $E$ to have analytic rank 0 or 1, based on a mod 4 congruence of 2-adic logarithms of Heegner points. As an application, we prove new cases of Silverman's conjecture that there exists a positive proposition of prime twists of $E$ of rank zero (resp. positive rank).
\end{abstract}

\setcounter{tocdepth}{1}

\section{Introduction}

\subsection{Silverman's conjecture}

Let $E/\mathbb{Q}$ be an elliptic curve. For a square-free integer $d$, we denote by $E^{(d)}/\mathbb{Q}$ its quadratic twist by $\mathbb{Q}(\sqrt{d})$. Silverman made the following conjecture concerning the prime twists of $E$ (see \cite[p.653]{Ono1998}, \cite[p.350]{Ono1997}).

\begin{conjecture}[Silverman]\label{conj:silverman}
  Let $E/\mathbb{Q}$ be an elliptic curve. Then there exists a positive proportion of primes $\ell$ such that $E^{(\ell)}$ or $E^{(-\ell)}$ has rank $r=0$ (resp. $r>0$).
\end{conjecture}

\begin{remark}
Conjecture \ref{conj:silverman} is known for the congruent number curve $E: y^2=x^3-x$. In fact, $E^{(\ell)}$ has rank $r=0$ if $\ell\equiv3\pmod{8}$ and $r=1$ if $\ell\equiv 5,7\pmod{8}$. This follows from classical 2-descent for $r=0$ and Birch \cite{Birch1970} and Monsky \cite{Monsky1990} for $r=1$ (see also \cite{Stephens1975}).
\end{remark}

\begin{remark}
  Although Conjecture \ref{conj:silverman} is still open in general, many special cases have been proved. For $r=0$, see Ono \cite{Ono1997} and Ono--Skinner \cite[Cor. 2]{Ono1998} (including all elliptic curves with conductor $\le100$). For $r=1$, see Coates--Y. Li--Tian--Zhai \cite[Thm. 1.1]{Coates2015}.
\end{remark}

In our recent work \cite[Thm. 3.3]{KrizLi2016a}, we have proved Conjecture \ref{conj:silverman} (for both $r=0$ and $r=1$) for a wide class of elliptic curves with $E(\mathbb{Q})[2]=0$. The goal of this short note is to extend our method to certain elliptic curves with $E(\mathbb{Q})[2]\cong \mathbb{Z}/2\mathbb{Z}$.

\subsection{Main results}

Let $E/\mathbb{Q}$ be an elliptic curve of conductor $N$. We will use $K$ to denote an imaginary quadratic field satisfying the \emph{Heegner hypothesis for $N$}:
\begin{center}
each prime factor $\ell$ of $N$ is split in $K$.
\end{center}
We denote by $P\in E(K)$ the corresponding Heegner point, defined up to sign and torsion with respect to a fixed modular parametrization $\pi_E: X_0(N)\rightarrow E$. Let $$f(q)=\sum_{n=1}^\infty a_n(E) q^n\in S_2^\mathrm{new}(\Gamma_0(N))$$ be the normalized newform associated  to $E$. Let $\omega_E\in \Omega_{E/\mathbb{Q}}^1 := H^0(E/\mathbb{Q},\Omega^1)$ such that $$\pi_E^*(\omega_E)= f(q) \cdot dq/q.$$ We denote by $\log_{\omega_E}$ the formal logarithm associated to $\omega_E$.

Our main result is the following criterion for prime twists of $E$ of analytic (and hence algebraic) rank 0 or 1. 

\begin{theorem}\label{thm:criterion}
Let $E/\mathbb{Q}$ be an elliptic curve. Assume $E(\mathbb{Q})[2]\cong \mathbb{Z}/2 \mathbb{Z}$ and $E$ has no rational cyclic 4-isogeny. Assume there exists an imaginary quadratic field $K$ satisfying the Heegner hypothesis for $N$ such that
  \begin{equation}
    \label{eq:star}\
2\text{ splits in } K \text{ and }\quad \frac{|\tilde E^\mathrm{ns}(\mathbb{F}_2)|\cdot\log_{\omega_E}(P)}{2}\not\equiv0\pmod{2}.     \tag{$\bigstar$}
\end{equation} Let $\mathcal{S}$ be the set of primes $$\mathcal{S}:=\{\ell\nmid 2N: \ell\text{ splits in } K, |E(\mathbb{F}_\ell)|\not\equiv 0\bmod{4} \}.$$ Let $\mathcal{N}$ be the set of signed primes $$\mathcal{N}=\{d=\pm \ell: \ell \in \mathcal{S}, \text{any odd prime }q||N \text{splits  in } \mathbb{Q}(\sqrt{d})\}.$$
Then for any $d\in \mathcal{\mathcal{N}}$, we have the analytic rank $r_\mathrm{an}(E^{(d)}/K)=1$. In particular, $$r_\mathrm{an}(E^{(d)}/\mathbb{Q})=
\begin{cases}
  0 , & \text{if }w(E^{(d)}/\mathbb{Q})=+1, \\
  1, & \text{if } w(E^{(d)}/\mathbb{Q})=-1.
\end{cases}$$
 where $w(E^{(d)}/\mathbb{Q})$ denotes the global root number of $E^{(d)}/\mathbb{Q}$.
\end{theorem}

\begin{remark}
  Recall that $|\tilde E^\mathrm{ns}(\mathbb{F}_\ell)|$ denotes the number of $\mathbb{F}_\ell$-points of the nonsingular part of the mod $\ell$ reduction of $E$, which is $|E(\mathbb{F}_\ell)|=\ell+1-a_\ell(E)$ if $\ell\nmid N$, $\ell\pm1$ if $\ell|| N$ and $\ell$ if $\ell^2|N$. 
\end{remark}

\begin{remark}
  The assumption on Heegner points in Theorem \ref{thm:criterion} forces $r_\mathrm{an}(E/\mathbb{Q})\le1$.
\end{remark}

As a consequence, we deduce the following cases of Silverman's conjecture.

\begin{theorem}\label{thm:mod4}
  Let $E/\mathbb{Q}$ as in Theorem \ref{thm:criterion}. Let $\phi: E\rightarrow E_0:=E/E(\mathbb{Q})[2]$ be the natural 2-isogeny. Assume the fields $\mathbb{Q}(E[2], E_0[2])$, $\mathbb{Q}(\sqrt{-N})$, $\mathbb{Q}(\sqrt{q})$ (where $q$ runs over odd primes $q||N$) are linearly disjoint.   Then Conjecture \ref{conj:silverman} holds for $E/\mathbb{Q}$.
\end{theorem}

\subsection{Novelty of the proof.} The proof of \cite[Thm. 3.3]{KrizLi2016a} mentioned above uses the mod 2 congruence between 2-adic logarithms of Heegner points on $E$ and $E^{(d)}$ (recalled in \S \ref{sec:congr-betw-heegn} below), arising from the isomorphism of Galois representations $E[2]\cong E^{(d)}[2]$. For the congruence to be nontrivial on both sides, one needs the extra factor $|E(\mathbb{F}_\ell)|$ appearing in the formula to be \emph{odd} for $\ell|d$. This is only possible when $E(\mathbb{Q})[2]=0$.

When $E(\mathbb{Q})[2]\ne0$, we instead take advantage of the exceptional isomorphism between the mod 4 semisimplified Galois representations $E[4]^\mathrm{ss}\cong E^{(d)}[4]^\mathrm{ss}$, and consequently a \emph{mod 4 congruence} between 2-adic logarithm of Heegner points. When $E(\mathbb{Q})[2]=\mathbb{Z}/2 \mathbb{Z}$ and $E$ has no rational cyclic 4-isogeny, it is possible that the extra factor $|E(\mathbb{F}_\ell)|$  is even but \emph{nonzero mod 4}. This is the key observation to prove Theorem \ref{thm:criterion}. The application Theorem \ref{thm:mod4} then follows by Chebotarev's density after translating the condition $|E(\mathbb{F}_\ell)|\not\equiv0\pmod{4}$ into an inert condition for $\ell$ in $\mathbb{Q}(E[2])$ and $\mathbb{Q}(E_0[2])$ (Lemma \ref{lem:mod4}).

\subsection{Acknowledgments}  The examples in this note are computed using Sage (\cite{sage}).

\section{Examples}

Let us illustrate the main results by two explicit examples.

\begin{example}
  Consider the elliptic curve (in Cremona's labeling) $$E=256b1: y^2=x^3-2x$$ with $E(\mathbb{Q})[2]\cong\mathbb{Z}/2 \mathbb{Z}$. It has $j$-invariant 1728 and CM by $\mathbb{Q}(i)$.  The imaginary quadratic field $K=\mathbb{Q}(\sqrt{-7})$ satisfies the Heegner hypothesis. The associated Heegner point $y_K=(-1,-1)$ satisfies Assumption (\ref{eq:star}). The set $\mathcal{S}$ consists of primes $\ell$ such that $\ell\equiv1,2,4\pmod{7}$ and $\ell\equiv5\pmod{8}$: $$\mathcal{S}=\{29,
 37,
 53,
 109,
 149,
 197,
 277,
 317,
 373,
 389,\ldots,\}.$$ By Theorem \ref{thm:criterion}, we have $$r_\mathrm{an}(E^{(\pm\ell)}/K)=1,\text{ for any } \ell\in\mathcal{S}.$$ We compute the global root number $w(E^{(\pm \ell)}/\mathbb{Q})=-1$ and conclude that $$r_\mathrm{an}(E^{(\pm \ell)}/\mathbb{Q})=1, \quad r_\mathrm{an}(E^{(\pm 7 \ell)}/\mathbb{Q})=0, \text{ for any } \ell\in \mathcal{S}.$$
\end{example}

\begin{remark}
  Notice the two congruence conditions for $\ell\in S$ are both necessary for the conclusion: for example, we have $r_\mathrm{an}(E^{(\ell)})=2$ for $\ell=31$ and $r_\mathrm{an}(E^{(7 \ell)})=2$ for $\ell=5$.
\end{remark}

\begin{example}
  Consider the elliptic curve $$E=256a1: y^2=x^3+x^2-3x+1$$ with $E(\mathbb{Q})[2]\cong \mathbb{Z}/2 \mathbb{Z}$. It has $j$-invariant 8000 and CM by $\mathbb{Q}(\sqrt{-2})$. The imaginary quadratic field $K=\mathbb{Q} (\sqrt{-7})$ satisfies the Heegner hypothesis. The associated Heegner point $y_K=(0,1)$ satisfies Assumption (\ref{eq:star}). The 2-isogenous curve is $$E_0=256a2: y^2=x^3+x^2-13x-21.$$ We have $\mathbb{Q}(E[2])=\mathbb{Q}(E_0[2])=\mathbb{Q}(\sqrt{2})$ and $\mathbb{Q}(\sqrt{-N})=\mathbb{Q}(i)$. Hence $\mathbb{Q}(E[2], E_0[2])$ and $\mathbb{Q}(\sqrt{-N})$ are linearly disjoint. Since there is no odd prime $q||N$, Theorem \ref{thm:mod4} implies that Silverman's conjecture holds for $E$.

  In fact, the set $\mathcal{S}$ in this case consists of primes $\ell$ such that $\ell\equiv1,2,4\pmod{7}$ and $\ell\equiv3,5\pmod{8}$: $$\mathcal{S}=\{11,
 29,
 37,
 43,
 53,
 67,
 107,
 109,
 149,
 163,
 179,
 197,
 211,
 277,
 317,
 331,\ldots\}.$$ Computing the global root number gives $$r_\mathrm{an}(E^{(\ell)}/\mathbb{Q})=1, \quad r_\mathrm{an}(E^{(-\ell)}/\mathbb{Q})=0,\text{ for any }\ell\in \mathcal{S}.$$
\end{example}

\section{Proof of Theorem \ref{thm:criterion}}

\subsection{Congruences between Heegner points}\label{sec:congr-betw-heegn} We first recall the main theorem of \cite{KrizLi2016a}. 
\begin{theorem}\label{thm:maincongruence} Let $E$ and $E'$ be two elliptic curves over $\mathbb{Q}$ of conductors $N$ and $N'$ respectively. Suppose $p$ is a prime such that there is an isomorphism of semisimplified  $G_\mathbb{Q}:=\Gal(\overline{\mathbb{Q}}/\mathbb{Q})$-representations $$E[p^m]^{\mathrm{ss}} \cong E'[p^m]^{\mathrm{ss}}$$ for some $m\ge1$. Let $K$ be an imaginary quadratic field satisfying the Heegner hypothesis for both $N$ and $N'$. Let $P \in E(K)$ and $P' \in E'(K)$ be the Heegner points. Assume $p$ is split in $K$. Then we have
$$\left(\prod_{\ell|pNN'/M}\frac{|\tilde{E}^{\mathrm{ns}}(\mathbb{F}_{\ell})|}{\ell}\right)\cdot \log_{\omega_E}P \equiv \pm\left(\prod_{\ell|pNN'/M}\frac{|\tilde{E}'^{,\mathrm{ns}}(\mathbb{F}_{\ell})|}{\ell}\right)\cdot\log_{\omega_{E'}}P' \pmod {p^m}.$$ Here 
$$M = \prod_{\ell|\gcd(N,N') \atop a_{\ell}(E)\equiv a_{\ell}(E')\pmod{p^m}}\ell^{\ord_{\ell}(NN')}.$$
\end{theorem}

\subsection{Proof of Theorem \ref{thm:criterion}}
For a prime $\ell\nmid Nd$, we have $a_\ell(E)=\pm a_\ell(E^{(d)})$ since $E^{(d)}$ is a quadratic twist of $E$. Since $E(\mathbb{Q})[2]\ne0$, we know that $|E(\mathbb{F}_\ell)|$ and $|E^{(d)}(\mathbb{F}_\ell)|$ are even since the reduction mod $\ell$ map is injective on prime-to-$\ell$ torsion. Hence if $\ell\ne2$, then $a_\ell(E)$, $a_\ell(E^{(d)})$ are also even. Since $a_\ell(E)=\pm a_\ell(E^{(d)})$, we obtain the following mod 4 congruence $$a_\ell(E)\equiv a_\ell(E^{(d)})\pmod{4},\quad \text{for any }\ell\nmid 2Nd.$$ It follows that we have an isomorphism of $G_\mathbb{Q}$-representations $$E[4]^\mathrm{ss}\cong E^{(d)}[4]^\mathrm{ss}.$$

Now we can apply Theorem \ref{thm:maincongruence} to $E'=E^{(d)}$, $p=2$ and $m=2$. By assumption, any prime $\ell| 2N$ splits in $K$. By the definition of $\mathcal{S}$, the prime $\ell=|d|$ splits in $K$. Notice the odd prime factors of $N'=N(E^{(d)})$ are exactly the odd prime factors of $Nd$, thus $K$ also satisfies the Heegner hypothesis for $N'$.

Let $\ell|\gcd(N,N')$ be an odd prime. We have:
\begin{enumerate}
\item if $\ell|| N$, then $a_\ell(E), a_\ell(E^{(d)})\in\{\pm1\}$ is determined by their local root numbers at $\ell$. By the definition of $\mathcal{N}$, we know that $\ell$ splits in $\mathbb{Q}(\sqrt{d})$, and hence $E/\mathbb{Q}_\ell$ and $E^{(d)}/\mathbb{Q}_\ell$ are isomorphic. It follows that $a_\ell(E)=a_\ell(E^{(d)})$.
\item if $\ell^2 |N$, then $a_\ell(E)=a_\ell(E^{(d)})=0$,
\end{enumerate}

Therefore $M$ is divisible by all the prime factors of $\gcd(N, N')$. Notice the odd part of $\gcd(N,N')$ equals to the odd part of $N$, so the congruence formula in Theorem \ref{thm:maincongruence} implies
\begin{equation}
  \label{eq:2d}
  \prod_{\ell|2d}\frac{|\tilde{E}^{\mathrm{ns}}(\mathbb{F}_{\ell})|}{\ell}  \cdot \log_{\omega_E}P \equiv \pm \prod_{\ell|2d}\frac{|\tilde{E}^{(d),\mathrm{ns}}(\mathbb{F}_{\ell})|}{\ell}\cdot \log_{\omega_{E^{(d)}}}P^{(d)}\pmod{4}.
\end{equation}
For $\ell=|d|$, we have $$|E(\mathbb{F}_\ell)|\not\equiv0\pmod{4}$$ by the definition of $\mathcal{S}$. Now Assumption (\ref{eq:star}) implies that the left-hand-side of (\ref{eq:2d}) is nonzero mod $4$. Hence the right-hand-side of (\ref{eq:2d}) is also nonzero. In particular, the Heegner point $P^{(d)}\in E^{(d)}(K)$ is non-torsion, and hence $r_\mathrm{an}(E^{(d)}/K)=1$ by the theorem of Gross--Zagier \cite{Gross1986} and Kolyvagin \cite{Kolyvagin1990}, \cite{Kolyvagin1988}, as desired.

\section{Proof of Theorem \ref{thm:mod4}} 
\subsection{Elliptic curves with partial 2-torsion and no rational cyclic 4-isogeny}

Let $E$ be an elliptic curve of conductor $N$. Assume $E(\mathbb{Q})[2]\cong \mathbb{Z}/2 \mathbb{Z}$. Then $\mathbb{Q}(E[2])/\mathbb{Q}$ is the quadratic extension $\mathbb{Q}(\sqrt{\Delta_E})$, where $\Delta_E$ is the discriminant of a Weierstrass equation of $E$.

Let $\phi: E\rightarrow E_0:=E/E(\mathbb{Q})[2]$ be the natural 2-isogeny. By \cite[Lem. 4.2 (i)]{Klagsbrun2017}, $E$ has no rational cyclic 4-isogeny if and only if $\mathbb{Q}(E_0[2])/\mathbb{Q}$ is a quadratic extension. Assume we are in this case, then $\mathbb{Q}(E_0[2])=\mathbb{Q}(\sqrt{\Delta_{E_0}})$.

\begin{lemma}\label{lem:mod4}
Let $\ell\nmid N$ be a prime. Then the following are equivalent:
  \begin{enumerate}
  \item   $|E(\mathbb{F}_\ell)|\not\equiv0\pmod{4}$,
  \item $E(\mathbb{F}_\ell)[2]\cong E_0(\mathbb{F}_\ell)[2]\cong \mathbb{Z}/2 \mathbb{Z}$,
  \item $\ell$ is inert in both $\mathbb{Q}(E[2])$ and $\mathbb{Q}(E_0[2])$.
  \end{enumerate}
\end{lemma}

\begin{proof}
  Since $E$ and $E_0$ are isogenous and $\ell$ is a prime of good reduction, we know that $|E(\mathbb{F}_\ell)|=|E_0(\mathbb{F}_\ell)|$. So $|E(\mathbb{F}_\ell)|\not\equiv0\pmod{4}$ if and only if $|E_0(\mathbb{F}_\ell)|\not\equiv0\pmod{4}$. In this case, certainly (2) holds. Conversely, if (2) holds, then $E(\mathbb{F}_\ell)[4]\cong \mathbb{Z}/2 \mathbb{Z}$ (otherwise $E(\mathbb{F}_\ell)[4]\cong \mathbb{Z}/4 \mathbb{Z}$, and thus $E_0(\mathbb{F}_\ell)[2]\cong \mathbb{Z}/2 \mathbb{Z}\times \mathbb{Z}/2 \mathbb{Z}$ generated by $\phi(E(\mathbb{F}_\ell)[4])$ and the kernel of the dual isogeny $\hat\phi:E_0\rightarrow E$), hence $|E(\mathbb{F}_\ell)|\not\equiv0\pmod{4}$. We have shown that (1) is equivalent to (2).
  
  Moreover, $E(\mathbb{F}_\ell)[2]\cong \mathbb{Z}/2 \mathbb{Z}$ (resp. $\mathbb{Z}/2 \mathbb{Z}\times \mathbb{Z}/2 \mathbb{Z}$) if and only if $\mathbb{Q}_\ell(E[2])/\mathbb{Q}_\ell$ is a quadratic extension (resp. the trivial extension), if and only if $\ell$ is inert (resp. split) in $\mathbb{Q}(E[2])$. Similarly we know that $E_0(\mathbb{F}_\ell)[2]\cong \mathbb{Z}/2 \mathbb{Z}$ if and only if  $\ell$ is inert in $\mathbb{Q}(E_0[2])$. It follows that (2) is equivalent to (3).
\end{proof}

\subsection{Proof of Theorem~\ref{thm:mod4}}
 By assumption, the fields $\mathbb{Q}(E[2], E_0[2])$, $\mathbb{Q}(\sqrt{q})$ ($q$ runs all odd prime $q||N$) are linearly disjoint. Since $K$ satisfies the Heegner hypothesis for $N$ and 2 splits in $K$, we know the discriminant $d_K$ of $K$ is coprime to $2N$, hence $K$ is also linearly disjoint from the fields $\mathbb{Q}(E[2],E_0[2])$ and $\mathbb{Q}(\sqrt{q})$'s. It follows from Chebotarev's density that there is a positive density set $\mathcal{T}$ of primes $\ell\nmid 2N$ such that
  \begin{enumerate}
  \item $\ell$ is split in $K$,
  \item $\ell$ is inert in both $\mathbb{Q}(E[2])$ and $\mathbb{Q}(E_0[2])$,
  \item $\ell$ is split in $\mathbb{Q}(\sqrt{q})$ for any odd prime $q||N$.
  \end{enumerate}

    By Lemma \ref{lem:mod4}, we know $\mathcal{T}\subseteq \mathcal{S}$. For $\ell\in \mathcal{T}$, we consider $d=\ell^*:=(-1)^{(\ell-1)/2}\ell$. By the quadratic reciprocity law, we know that odd $q||N$ is split in $\mathbb{Q}(\sqrt{\ell^*})$ if and only if $\ell$ is split in $\mathbb{Q}(\sqrt{q})$. In particular, for any $\ell\in \mathcal{T}$, we have $\ell^*\in \mathcal{N}$. Now Theorem \ref{thm:criterion} implies that $r_\mathrm{an}(E^{(\ell^*)}/K)=1$. Moreover, $$r_\mathrm{an}(E^{(\ell^*)}/\mathbb{Q})=
    \begin{cases}
      0, & w(E^{(\ell^*)}/\mathbb{Q})=+1,\\
      1, & w(E^{(\ell^*)}/\mathbb{Q})=-1.
    \end{cases}
    $$

    Since $\mathbb{Q}(\sqrt{\ell^*})$ has discriminant coprime to $2N$, we have the well known formula $$w(E^{(\ell^*)}/\mathbb{Q})=w(E/\mathbb{Q})\cdot \legendre{\ell^*}{-N}.$$ By the quadratic reciprocity law, we obtain $$w(E^{(\ell^*)}/\mathbb{Q})=w(E/\mathbb{Q})\cdot \legendre{-N}{\ell}.$$ By assumption, $\mathbb{Q}(\sqrt{-N})$ is also linearly disjoint from the fields considered above, hence the global root number $w(E^{(\ell^*)}/\mathbb{Q})$ takes both signs for a positive proportion of $\ell\in \mathcal{T}$ by Chebotarev's density. Therefore $r_\mathrm{an}(E^{(\ell^*)}/\mathbb{Q})$ takes both values 0 and 1 for a positive proportion of $\ell\in \mathcal{T}$, as desired.

\bibliographystyle{alpha}
\bibliography{Congruence}

\end{document}